\documentclass[11pt]{amsart}

\usepackage[T1]{fontenc}
\usepackage{lmodern}
\usepackage{microtype}
\usepackage{amsmath,amssymb,mathtools}
\usepackage{enumitem}
\usepackage[colorlinks=true,linkcolor=blue,citecolor=blue,urlcolor=blue]{hyperref}
\usepackage[nameinlink,capitalise,noabbrev]{cleveref}

\newtheorem{theorem}{Theorem}[section]
\newtheorem{proposition}[theorem]{Proposition}
\newtheorem{lemma}[theorem]{Lemma}

\newtheorem{observation}[theorem]{Observation}
\theoremstyle{remark}
\newtheorem{remark}[theorem]{Remark}
\newtheorem{question}[theorem]{Question}

\newcommand{\F}{\mathbb F}
\newcommand{\Z}{\mathcal Z}
\newcommand{\Bas}{\mathcal B}
\newcommand{\bn}{\operatorname{bn}}
\newcommand{\eg}{\operatorname{eg}}
\newcommand{\ch}{\operatorname{ch}}
\newcommand{\cng}{\operatorname{cong}}

\usepackage{thmtools,thm-restate}

\title{Logarithmic basis number of graphs and regular matroids}
\author{Kolja Knauer}
\date{}
\address{Departament de Matemàtiques i Informàtica, Universitat de Barcelona, Barcelona, Spain, Centre de Recerca Matemàtica (CRM), Campus de Bellaterra, Edifici C, 08193 Bellaterra, Barcelona, Spain}
\email{kolja.knauer@ub.edu}

\begin{document}

\nocite{BanksSchmeichel1982,BazarganiEtAl2026,ChekuriVondrakZenklusen2010,FreedmanHastings2021,GenietGiocanti2026,GurjarVishnoi2021,Heller1957,KavithaEtAl2009,MacLane1937,MacWilliamsSloane1977,LehnerMiraftab2026,Longyear1980,MilgramUngar1977,MiraftabMorinYuditsky2026,MoharThomassen2001,RichterShank1984,Rizzi2009,Schmeichel1981,Seymour1980,Varshamov1957,WangIrani2026,Wargo1996,Weiss1984}

\begin{abstract}
The basis number $\bn(G)$ of a graph $G$ is the minimum edge-congestion of a basis of its cycle space. We prove that every finite $n$-vertex multigraph satisfies $\bn(G)=O(\log n)$, resolving, for simple graphs, a question of Bazargani, Biedl, Bose, Maheshwari and Miraftab, subsequently stated as a conjecture by Miraftab, Morin and Yuditsky. The argument also yields the cycle-rank refinement $\bn(G)=O(\log \beta(G))$, where $\beta(G)$ is the dimension of the cycle space, and a reduction of Lehner and Miraftab, based on a theorem of Richter and Shank, then gives $\bn(G)=O(\log g)$ for graphs of Euler genus $g$. For regular matroids we prove the ground-set bound $\bn(M)=O(\log m)$, where $m=|E(M)|$, and logarithmic bounds in both the rank $r(M)$ and the cycle-space dimension $d=\dim(\Z(M))$. All these orders are best possible.
\end{abstract}

\maketitle

\section{Introduction}\label{sec:introduction}

Let $G=(V,E)$ be a finite graph and let $\Z(G)\subseteq\F_2^E$ denote its cycle space, the vector space of Eulerian edge sets. A \emph{cycle basis} is a basis of $\Z(G)$. If $\Bas$ is a family of subgraphs and $e\in E$, define
\[
\begin{aligned}
\ch_{\Bas}(e)&:=\bigl|\{B\in\Bas:e\in B\}\bigr|,\\
\cng(\Bas)&:=\max_{e\in E}\ch_{\Bas}(e).
\end{aligned}
\]
We set $\cng(\Bas)=0$ when $E=\varnothing$. We call $\cng(\Bas)$ the \emph{edge-congestion}, or simply the \emph{congestion}, of $\Bas$. The \emph{basis number} of $G$, introduced by Schmeichel~\cite{Schmeichel1981}, is
\[
\bn(G):=\min\{\cng(\Bas):\Bas\text{ is a cycle basis of }G\}.
\]
Mac Lane's planarity criterion~\cite{MacLane1937} says that $G$ is planar if and only if $\bn(G)\le2$.

The terminology in the recent literature is not completely uniform. Geniet and Giocanti~\cite{GenietGiocanti2026} use \emph{edge-congestion}; Bazargani et al.~\cite{BazarganiEtAl2026} call $\ch_{\Bas}(e)$ the \emph{charge} of $e$; Miraftab, Morin and Yuditsky~\cite{MiraftabMorinYuditsky2026} use \emph{ply}; and Wang and Irani~\cite{WangIrani2026} use \emph{maximum edge participation}. We use \emph{congestion} for the maximum over edges and \emph{charge} for the quantity $\ch_{\Bas}(e)$ of a single edge.

Prior to this work, the best general bound was $O(\log^2 n)$. Freedman and Hastings~\cite[Appendix~A]{FreedmanHastings2021}, in work motivated by quantum codes, gave a randomized decongestion construction yielding this bound for simple $n$-vertex graphs; it is recorded explicitly as Theorem~1 by Miraftab, Morin and Yuditsky~\cite{MiraftabMorinYuditsky2026}. Bazargani, Biedl, Bose, Maheshwari and Miraftab~\cite[Section~5]{BazarganiEtAl2026} asked whether every $n$-vertex graph has basis number $O(\log n)$, and Miraftab, Morin and Yuditsky subsequently stated precisely this as Conjecture~12 in~\cite{MiraftabMorinYuditsky2026}.

For bounded genus, Lehner and Miraftab~\cite{LehnerMiraftab2026} proved $O(\log^2 g)$ and recorded an $\Omega(\log g)$ lower bound, leaving only the exponent of the logarithm open; our genus result closes this gap. Geniet and Giocanti~\cite{GenietGiocanti2026} use the bounded-genus theorem in their excluded-minor result. Replacing $O(\log^2 g)$ by $O(\log g)$ does not improve their stated polynomial dependence on the excluded minor. We believe that replacing the $O(\log^2 n)$ bound used by Miraftab, Morin and Yuditsky~\cite[Lemma~9 and Theorem~3]{MiraftabMorinYuditsky2026} with our $O(\log n)$ bound could also sharpen some of their estimates for path decompositions.

Our first main result resolves the general $n$-vertex problem.

\begin{theorem}\label{thm:main-n}
There is an absolute constant $C$ such that every finite multigraph $G$ on $n\ge2$ vertices satisfies
$\bn(G)\le C\log n$.
\end{theorem}

This yields a refinement for graphs of small cycle rank. Write $\beta(G):=\dim_{\F_2}\Z(G)=|E(G)|-|V(G)|+c(G)$ for the cycle rank, or first Betti number, where $c(G)$ is the number of connected components.

\begin{theorem}\label{thm:main-beta}
There is an absolute constant $C_\beta$ such that every finite multigraph $G$ with $\beta(G)\ge2$ satisfies
$\bn(G)\le C_\beta\log \beta(G)$.
\end{theorem}

A reduction of Lehner and Miraftab, based on a theorem of Richter and Shank, turns \cref{thm:main-beta} into the optimal genus bound.

\begin{theorem}\label{thm:main-genus}
There is an absolute constant $C_g$ such that every graph $G$ of Euler genus $g\ge2$ satisfies
$\bn(G)\le C_g\log g$.
\end{theorem}

\subsection*{The new ingredient}
Rizzi~\cite{Rizzi2009} proved that every graph with nonnegative edge weights of total weight $W$ has a cycle basis of total weight $O(W\log n)$; see also~\cite[Theorem~4.4]{KavithaEtAl2009}. For every cycle basis $\Bas$, we have
\[
\sum_{C\in\Bas}w(C)=\sum_{e\in E}w(e)\ch_{\Bas}(e).
\]
Thus Rizzi's theorem gives a basis of logarithmic weighted average charge for every probability weighting of the edges. Linear programming duality converts these separate statements into a distribution on cycle bases with logarithmic expected charge at every edge.

The marginals of this distribution form a point in the base polytope of the linear matroid whose bases are the cycle bases of $G$. Swap rounding of Chekuri, Vondr\'{a}k and Zenklusen~\cite{ChekuriVondrakZenklusen2010} preserves these marginals and supplies Chernoff concentration, yielding one basis with logarithmic charge at every edge. A shorter route through Freedman and Hastings is explained in \cref{rem:fh-shortcut}.

After the first version appeared on arXiv~\cite{Knauer2026v1}, Long, Pettie and Saranurak~\cite[Theorem~3.2 and Section~4]{LongPettieSaranurak2026} made our Rizzi-based proof of \cref{thm:main-n} algorithmic using multiplicative weights and derandomized swap rounding. Their deterministic polynomial-time construction of a cycle basis of congestion $O(\log n)$ yields edge-fault-tolerant connectivity labeling schemes with $O(\log^2 n)$-bit labels~\cite[Theorem~3.1]{LongPettieSaranurak2026}.

\subsection*{Regular matroids}
Finally, the cycle-rank bound extends from graphic to regular matroids. If $M$ is a finite binary matroid, write $\Z(M)\subseteq\F_2^{E(M)}$ for its cycle space, whose elements are the disjoint unions of circuits. Define $\bn(M)$ analogously as the minimum element-congestion of a basis of $\Z(M)$ consisting of circuits; circuit bases of binary matroids are classical objects, see Longyear~\cite{Longyear1980}, and for regular matroids Wargo~\cite{Wargo1996}. Welsh's matroidal form of Mac Lane's theorem~\cite{Welsh1969} states that a binary matroid $M$ is cographic if and only if $\bn(M)\le2$. We first obtain a logarithmic bound in the size of the ground set.

\begin{theorem}\label{thm:main-regular-set}
There is an absolute constant $C_E$ such that every regular matroid $M$ on $m\ge2$ elements satisfies
$\bn(M)\le C_E\log m$.
\end{theorem}

This can be strengthened from the number of elements to the dimension of the cycle space.

\begin{theorem}\label{thm:main-regular}
There is an absolute constant $C_{\rm dim}$ such that every regular matroid $M$ with $d:=\dim(\Z(M))\ge2$ satisfies
$\bn(M)\le C_{\rm dim}\log d$.
\end{theorem}

Simplification also yields a logarithmic bound in terms of rank.

\begin{restatable}{theorem}{regularranktheorem}\label{thm:main-regular-rank}
There is an absolute constant $C_{\rm rk}$ such that every regular matroid $M$ of rank $r:=r(M)\ge2$ satisfies $\bn(M)\le C_{\rm rk}\log r$.
\end{restatable}

The proofs are given in \cref{sec:regular}. The ground-set bound combines \cref{thm:main-beta} with the unordered form of Seymour's decomposition theorem. Its main ingredient is a priority version of \cref{thm:main-beta}: after ordering the elements of a graphic matroid, the $j$th element can be given charge $O(1+\log j)$. A final reduction to a cosimple matroid and Heller's extremal theorem then converts \cref{thm:main-regular-set} into the dimension bound in \cref{thm:main-regular}. Simplification gives the corresponding rank bound in \cref{thm:main-regular-rank}. We note in \cref{sec:binary} that an extension of our results to general binary matroids is not possible.

\medskip

The matching lower-bound orders are classical. The girth bound of Banks and Schmeichel~\cite{BanksSchmeichel1982}, combined with any standard family of cubic graphs of logarithmic girth, for instance the bipartite sextet graphs of Weiss~\cite{Weiss1984}, gives $\Omega(\log n)$ and $\Omega(\log\beta)$. Lehner and Miraftab~\cite{LehnerMiraftab2026} explicitly obtain the logarithmic genus lower bound. Since graphic matroids are regular, the same family also shows that \cref{thm:main-regular-set,thm:main-regular,thm:main-regular-rank} are sharp in order.

\section{Cycle bases and two external ingredients}\label{sec:ingredients}

We allow loops and parallel edges; a loop contributes two to the degree of its incident vertex. We write $\log$ for the base-$2$ logarithm and $\ln$ for the natural logarithm.

Some authors allow arbitrary Eulerian subgraphs as basis elements, while others require cycles. This distinction does not affect the parameter: if a basis element $A$ is not a cycle, decompose $A$ into edge-disjoint cycles and replace it by one of these cycles that is not in the span of the remaining basis elements. Iterating gives a cycle basis without increasing any edge charge.

For an edge weighting $w$, put $w(C)=\sum_{e\in C}w(e)$.

\begin{theorem}[Rizzi]\label{thm:rizzi}
There is an absolute constant $a>0$ such that every simple graph $G$ on $n\ge2$ vertices, with nonnegative edge weights of total weight $W$, has a cycle basis $\Bas$ satisfying $\sum_{C\in\Bas}w(C)\le aW\log n$.
\end{theorem}

Rizzi~\cite{Rizzi2009} proves the stronger statement that the basis may be chosen weakly fundamental; see also~\cite[Theorem~4.4]{KavithaEtAl2009}. We use only the weight bound.

For a finite matroid $M$ on $S$, let $P_M:=\operatorname{conv}\{\mathbf 1_B:B\text{ is a base of }M\}$, where $\mathbf 1_B$ is the incidence vector of $B$. A point $x\in P_M$ is a convex combination of incidence vectors of bases. Swap rounding produces a random base $B$ with $\Pr(s\in B)=x_s$ for every $s\in S$ and gives Chernoff upper tails for the number of selected elements in any prescribed subset of $S$~\cite[Theorem~1.1 and Corollary~1.2]{ChekuriVondrakZenklusen2010}.

\begin{lemma}[Simultaneous rounding]\label{lem:rounding}
There is an absolute constant $c>0$ with the following property. Let $M$ be a finite matroid on $S$, let $x\in P_M$, and let $A_1,\ldots,A_q\subseteq S$, where $q\ge2$. If $x(A_i):=\sum_{s\in A_i}x_s\le L$ for every $i$, then $M$ has a base $B$ such that $|B\cap A_i|\le c(L+\ln q)$ for all $i$.
\end{lemma}

\begin{proof}
Apply swap rounding and put $X_i:=|B\cap A_i|$ and $\mu_i:=\mathbb E[X_i]=x(A_i)\le L$. If $\mu_i=0$, then $X_i=0$ almost surely; otherwise the Chernoff bound gives $\Pr[X_i\ge t]\le(e\mu_i/t)^t$ for $t\ge\mu_i$~\cite[Corollary~1.2]{ChekuriVondrakZenklusen2010}. Taking $t=6(L+\ln q)$ yields $\Pr[X_i\ge t]\le(e/6)^t\le q^{-2}$. The union bound proves the claim with $c=6$.
\end{proof}

\section{From weighted average charge to low congestion}\label{sec:marginals}

We first treat connected simple graphs. For such a graph $G$, let $\mathfrak B(G)$ be the finite set of all simple-cycle bases of $G$.

\begin{lemma}[Linear-programming duality]\label{lem:minimax}
Let $G$ be a connected simple graph on $n\ge2$ vertices. There is a probability distribution $p=(p_{\Bas})_{\Bas\in\mathfrak B(G)}$ such that $\sum_{\Bas\in\mathfrak B(G)}p_{\Bas}\ch_{\Bas}(e)\le a\log n$ for every $e\in E(G)$, where $a$ is the constant from \cref{thm:rizzi}.
\end{lemma}

\begin{proof}
For an edge $e$ and a basis $\Bas$, set $a_{e,\Bas}:=\ch_{\Bas}(e)$. Linear programming duality gives
\begin{equation}\label{eq:minimax}
\max_w\min_{\Bas\in\mathfrak B(G)}\sum_e w_ea_{e,\Bas}
=\min_p\max_e\sum_{\Bas\in\mathfrak B(G)}p_{\Bas}a_{e,\Bas},
\end{equation}
where $w$ ranges over probability distributions on $E(G)$ and $p$ over probability distributions on $\mathfrak B(G)$. For every $w$, \cref{thm:rizzi} gives a basis with $\sum_e w_ea_{e,\Bas}=\sum_{C\in\Bas}w(C)\le a\log n$. Hence the left-hand side of \eqref{eq:minimax} is at most $a\log n$, and the right-hand side yields the required distribution.
\end{proof}

\begin{remark}\label{rem:fh-shortcut}
After the first arXiv version~\cite{Knauer2026v1} appeared, a shorter route to \cref{lem:minimax} was brought to our attention. The construction of Freedman and Hastings~\cite[proof of Lemma~A.0.2]{FreedmanHastings2021} already gives a distribution with expected charge $O(\log n)$ at every edge. Indeed, fix $e\in E(G)$ and put $f=e$. Until $f$ is deleted, replace it by the new edge whenever a degree-two vertex incident with $f$ is suppressed. Each Case~3 cycle containing $f$ deletes $f$ with conditional probability $\Omega(1/\log n)$, while Case~2A contributes at most one occurrence, since it also deletes $f$. This estimate needs no upper degree bound. Their distribution can therefore be used directly in \cref{prop:simple-connected}, bypassing the weighted theorem and linear programming duality. We retain the original proof because it has already led to the algorithmic application described in the introduction.
\end{remark}

We now use the matroid structure of cycle bases. Let $\mathcal S(G)$ be the set of simple cycles of $G$, represented by their incidence vectors in $\Z(G)$. Linear independence over $\F_2$ defines a vector matroid $M_G$ on ground set $\mathcal S(G)$, whose bases are precisely the simple-cycle bases of $G$; see~\cite[Theorem~3.10]{KavithaEtAl2009}.

\begin{proposition}\label{prop:simple-connected}
There is an absolute constant $C_0$ such that every connected simple graph $G$ on $n\ge2$ vertices satisfies
$\bn(G)\le C_0\log n$.
\end{proposition}

\begin{proof}
If $G$ is a tree there is nothing to prove, so assume that $G$ contains a cycle. Let $\Bas$ be a random cycle basis sampled from the distribution in \cref{lem:minimax}, and for every $C\in\mathcal S(G)$ put $x_C:=\Pr(C\in\Bas)$. Since $\Bas$ is always a base of $M_G$, $x=\mathbb E[\mathbf 1_{\Bas}]\in P_{M_G}$. For each graph edge $e$, let $A_e:=\{C\in\mathcal S(G):e\in C\}$. Then
\[
x(A_e)=\sum_{C\ni e}x_C=\mathbb{E}[\ch_{\Bas}(e)]\le a\log n.
\]
Write $m=|E(G)|$. Apply \cref{lem:rounding} to $M_G$ and the sets $(A_e)_{e\in E(G)}$. We obtain a base $\Bas'$ of $M_G$, hence a simple-cycle basis of $G$, with $\ch_{\Bas'}(e)=O(\log n+\ln m)$ for every $e\in E(G)$. Since $G$ is simple, $m\le\binom n2$, and therefore $\ln m=O(\log n)$.
\end{proof}

The extension to arbitrary multigraphs is standard. Bazargani et al.~\cite[Corollary~3.5]{BazarganiEtAl2026} show that adding a parallel copy of an edge increases the basis number to at most the maximum of the old basis number and $2$.

\begin{lemma}[Loops and parallel classes]\label{lem:parallel}
Let $G$ be a finite multigraph, and let $G_{\rm s}$ be obtained by deleting all loops and retaining one representative from every non-loop parallel class. Then
$\bn(G)\le\max\{\bn(G_{\rm s}),2\}$.
\end{lemma}

\begin{proof}
Repeatedly apply~\cite[Corollary~3.5]{BazarganiEtAl2026} to restore the deleted members of each parallel class. Restoring a loop adds an independent one-edge cycle, so the congestion becomes at most the maximum of its previous value and~$1$.
\end{proof}

\begin{proof}[Proof of \cref{thm:main-n}]
Let $G$ have $n\ge2$ vertices and form the simple graph $G_{\rm s}$ from \cref{lem:parallel}. The cycle space of $G_{\rm s}$ is the direct sum of the cycle spaces of its connected components, so its basis number is the maximum of their basis numbers. By \cref{prop:simple-connected}, every component with a cycle has basis number at most $C_0\log n$, while acyclic components have basis number zero. Thus $\bn(G_{\rm s})=O(\log n)$, and \cref{lem:parallel} gives $\bn(G)=O(\log n)$.
\end{proof}

\section{Cycle rank and genus}\label{sec:beta-genus}

We derive \cref{thm:main-beta} from \cref{thm:main-n} and then deduce \cref{thm:main-genus}. Bridges occur in no element of the cycle space and may be deleted without changing the basis number. The other reduction is suppression of degree-two vertices.

\begin{lemma}[Series reduction]\label{lem:series}
Let $v$ be a degree-two vertex incident with two distinct non-loop edges $uv$ and $vw$, where possibly $u=w$. Let $G'$ be obtained by suppressing $v$ and replacing those two edges by one edge $uw$ (a loop if $u=w$). Then
$\bn(G')=\bn(G)$ and $\beta(G')=\beta(G)$.
\end{lemma}

\begin{proof}
The equality of basis numbers is exactly invariance under edge subdivision~\cite[Lemma~3.6]{BazarganiEtAl2026}. The equality of cycle ranks follows immediately from $|E(G')|=|E(G)|-1$, $|V(G')|=|V(G)|-1$, and $c(G')=c(G)$.
\end{proof}

\begin{proof}[Proof of \cref{thm:main-beta}]
Put $r=\beta(G)\ge2$ and delete all bridges. It suffices to treat one connected component $K$ containing a cycle, of cycle rank $r_K\le r$. Suppress every degree-two vertex to which \cref{lem:series} applies, obtaining a multigraph $H$. By that lemma, $\bn(K)=\bn(H)$ and $\beta(H)=r_K$. If $|V(H)|=1$, then $H$ consists only of loops and $\bn(H)\le1$. Otherwise $H$ has minimum degree at least three. Writing $n_H=|V(H)|$ and $m_H=|E(H)|$, we have $r_K=m_H-n_H+1$ and $3n_H\le2m_H$, and hence $n_H\le2r_K-2\le2r-2$. By \cref{thm:main-n}, $\bn(K)=O(\log n_H)=O(\log r)$. Taking the maximum over the connected components containing cycles proves the theorem.
\end{proof}

For a closed surface $\Sigma$, let $\chi(\Sigma)$ denote its Euler characteristic and put $\eg(\Sigma):=2-\chi(\Sigma)$. Thus $\eg(\Sigma)=2h$ for an orientable surface of genus $h$ and $\eg(\Sigma)=h$ for a non-orientable surface of genus $h$. The Euler genus $\eg(G)$ of a graph is the minimum $\eg(\Sigma)$ over surfaces in which $G$ embeds. An embedding is \emph{$2$-cell}, or \emph{cellular}, if every connected component of $\Sigma\setminus G$ is homeomorphic to an open disc.

By~\cite[Propositions~3.4.1 and 3.4.2]{MoharThomassen2001}, every connected graph has a cellular embedding in a surface of minimum Euler genus.

The following topological reduction is due to Lehner and Miraftab~\cite[Lemma~7]{LehnerMiraftab2026}. It follows from their Lemma~5, which they derive from the main result of Richter and Shank~\cite{RichterShank1984}.

\begin{lemma}[Lehner--Miraftab, after Richter--Shank]\label{lem:LM}
Let $G$ be a connected graph with a $2$-cell embedding in a surface $\Sigma$. Then $G$ contains a subgraph $H$ satisfying
$\beta(H)=2-\chi(\Sigma)$ and $\bn(G)\le\bn(H)+2$.
\end{lemma}

\begin{proof}[Proof of \cref{thm:main-genus}]
Assume first that $G$ is connected and put $g=\eg(G)\ge2$. Choose a cellular embedding in a surface $\Sigma$ of Euler genus $g$. Then $2-\chi(\Sigma)=g$, and \cref{lem:LM} gives a subgraph $H$ with $\beta(H)=g$ and $\bn(G)\le\bn(H)+2$. Now \cref{thm:main-beta} gives $\bn(G)=O(\log g)$.

For disconnected graphs, take cycle bases componentwise. Components of Euler genus at most one have basis number at most three by \cref{lem:LM}, since the corresponding subgraph has cycle rank at most one, while every other component has Euler genus at most $g$; hence the same bound follows. If $h$ is the orientable genus, then $\eg(G)\le2h$; if $h$ is the non-orientable genus, then $\eg(G)\le h$. Hence the same logarithmic asymptotic bound holds for either convention.
\end{proof}

\begin{remark}\label{rem:no-recursion}
Lehner and Miraftab use \cref{lem:LM} recursively: after obtaining $H$ of prescribed cycle rank, they bound the genus of $H$ using the theorem of Milgram and Ungar~\cite{MilgramUngar1977} and iterate, leading to $O(\log^2 g)$.
\end{remark}

\section{Regular matroids}\label{sec:regular}

We now prove \cref{thm:main-regular-set,thm:main-regular,thm:main-regular-rank}. We shall occasionally use arbitrary nonzero cycles as basis elements. As observed above for graphs, the same replacement argument works verbatim for binary matroids: every basis of $\Z(M)$ can be replaced by a circuit basis without increasing the charge of any element.

We first derive a priority form of \cref{thm:main-beta}, whose only interesting part is the graphic case. The statement is immediate for cographic matroids and for $R_{10}$.

\begin{lemma}\label{lem:priority}
There is an absolute constant $C_p$ with the following property. Let $B$ be graphic, cographic, or isomorphic to $R_{10}$, and order its elements as $x_1,x_2,\ldots,x_m$. Then $\Z(B)$ has a basis $\Bas$ such that
$\ch_{\Bas}(x_j)\le C_p(1+\log j)$ for $1\le j\le m$.
\end{lemma}

\begin{proof}
First suppose that $B$ is the cycle matroid of a multigraph $G$. For $s\ge0$ put $N_s=2^{2^s}$ and set $N_{-1}=0$. Let $P_s:=\{x_j:N_{s-1}<j\le N_s\}\subseteq E$, truncating the last nonempty block. Put $G_{-1}=G$ and, for $s\ge0$, let $G_s$ be obtained from $G$ by deleting $P_0\cup\cdots\cup P_s$. Choose a spanning forest $F_s$ of $G_s$ and put $H_s=F_s\cup P_s$. Then
\begin{equation}\label{eq:priority-direct-sum}
\Z(G_{s-1})=\Z(G_s)\oplus\Z(H_s).
\end{equation}
Indeed, the intersection is trivial because it is supported on the forest $F_s$. 

To obtain that it is spanning, 
take \(C\in\mathcal Z(G_{s-1})\). For each edge \(e\) of \(C\) lying in \(G_s\setminus F_s\), its fundamental cycle \(C_e\) consists of \(e\) together with the unique path in \(F_s\) joining its endpoints.
Replacing \(C\) by \(C\oplus C_e\) cancels \(e\). Every other affected edge belongs to \(F_s\), so this introduces no other edge of \(G_s\setminus F_s\).
After repeating this, the remaining cycle vector \(C'\) is supported on \(F_s\cup P_s=H_s\). Thus the original vector decomposes as $$ C=C'\oplus D, \qquad C'\in\mathcal Z(H_s),\quad D\in\mathcal Z(G_s),$$
where \(D\) is the sum of the fundamental cycles used. 

Since adding \(|P_s|\) edges to a forest creates at most \(|P_s|\) independent cycles, we have $\dim(\Z(H_s))\le|P_s|\le N_s$. If this dimension is at most one, $H_s$ has a cycle basis of congestion at most one; otherwise \cref{thm:main-beta} gives congestion $O(\log N_s)=O(2^s)$. Thus the same asymptotic bound holds in both cases. Iterating \eqref{eq:priority-direct-sum} shows that the union of these bases is a basis of $\Z(G)$. If $x_j\in P_t$, then $x_j$ occurs only in $H_0,\ldots,H_t$, and hence has charge
\[
O\!\left(\sum_{s=0}^t2^s\right)=O(2^t)=O(1+\log j).
\]
If $B=M^*(G)$ is cographic, the vertex cuts $\delta(v)$, with one vertex omitted from each connected component of $G$, form a basis of $\Z(B)$ in which every element has charge at most two. Finally, $R_{10}$ is fixed and is taken care of by increasing $C_p$.
\end{proof}

Let $M_1,M_2$ be binary matroids on ground sets $E_1,E_2$, and put $S=E_1\cap E_2$. Their sum is the binary matroid $M$ on $(E_1\cup E_2)\setminus S$ whose cycles are precisely the sets $C_1\oplus C_2$ with $C_i\in\Z(M_i)$ for $i=1,2$ and $C_1\cap S=C_2\cap S$. Here $\oplus$ denotes symmetric difference, so the common elements cancel. Following~\cite[Definitions~3.12--3.13]{GurjarVishnoi2021}, this is called
\begin{enumerate}[label=(\roman*),nosep]
\item a \emph{$1$-sum} if $S=\varnothing$;
\item a \emph{$2$-sum} if $S=\{e\}$, the element $e$ is neither a loop nor a coloop in either summand, and $|E_1|,|E_2|\ge3$;
\item a \emph{$3$-sum} if $S$ is a three-element circuit of both summands, no cocircuit of either summand is contained in $S$, and $|E_1|,|E_2|\ge7$.
\end{enumerate}
In these cases we write $M=M_1\oplus_k M_2$, where $k\in\{1,2,3\}$.

\begin{lemma}\label{lem:star}
Let $B,M_1,\ldots,M_t$ be binary matroids and $M$ obtained by $1$-, $2$-, or $3$-sums of the $M_1,\ldots,M_t$ to $B$ along pairwise disjoint sets $S_i=E(B)\cap E(M_i)$, where $|S_i|\in\{0,1,3\}$. Assume that $E(M_i)\cap E(M_j)=\varnothing$ for $i\ne j$. 
Let $\Bas$ be a basis of $\Z(B)$ and, for every $i$, let $\Bas_i$ be a basis of $\Z(M_i\setminus S_i)$. Put $q_i:=\max_{s\in S_i}\ch_{\Bas}(s)$ when $S_i\ne\varnothing$, and $q_i=0$ otherwise. Then $\Z(M)$ has a basis in which every element of $E(B)\setminus\bigcup_iS_i$ has charge at most its charge in $\Bas$, while every element of $E(M_i)\setminus S_i$ has charge at most
$\cng(\Bas_i)+3q_i$.
\end{lemma}

\begin{proof}
For each $i$, the map $\Z(M_i)\to\F_2^{S_i}$, sending a cycle to its intersection with $S_i$, is onto. This is clear for $S_i=\varnothing$. If $|S_i|=1$, the common element is not a coloop, so it belongs to a circuit. If $|S_i|=3$ and the image were proper, a nonzero vector orthogonal to the image would give a cocycle supported on $S_i$, and hence a cocircuit contained in $S_i$, contrary to the definition of a $3$-sum.

For every $A\in\Bas$ and every $i$, choose a cycle $D_i(A)\in\Z(M_i)$ with  $D_i(A)\cap S_i=A\cap S_i$, taking $D_i(A)=\varnothing$ when this intersection is empty. Then, $\widehat A:=A\oplus D_1(A)\oplus\cdots\oplus D_t(A)$ is a cycle of $M$.

The family consisting of all $\widehat A$, together with the bases $\Bas_i$, spans $\Z(M)$. Indeed, any $C\in\Z(M)$ can be written $C=C_B\oplus C_1\oplus\cdots\oplus C_t$, where $C_B\in\Z(B)$, $C_i\in\Z(M_i)$, and $C_i\cap S_i=C_B\cap S_i$. Since $\Bas$ is a basis, there is a subfamily $\mathcal A\subseteq\Bas$ with $C_B=\bigoplus_{A\in\mathcal A}A$. For each $i$, put $R_i:=C_i\oplus\bigoplus_{A\in\mathcal A}D_i(A)$. This is a cycle of $M_i$. Moreover, the sum of the $D_i(A)$ has the same intersection with $S_i$ as $C_B$, and hence as $C_i$. These elements therefore cancel, so $R_i\cap S_i=\varnothing$. Thus $R_i\in\Z(M_i\setminus S_i)$ and can be expressed using the basis $\Bas_i$. Finally,
\[
C=\bigoplus_{A\in\mathcal A}\widehat A\oplus R_1\oplus\cdots\oplus R_t,
\]
which proves the spanning assertion. Choose a basis $\Bas'\subseteq\{\widehat A:A\in\Bas\}\cup\bigcup_{i=1}^t\Bas_i$.

We now bound the charges in $\Bas'$. If $e\in E(B)\setminus\bigcup_iS_i$, then none of the cycles $D_i(A)$ or the members of $\Bas_i$ contains $e$. Consequently, $e$ occurs in $\widehat A$ precisely when it occurs in $A$. Hence, its charge in $\Bas'$ is at most $\ch_{\Bas}(e)$. Next, let $e\in E(M_i)\setminus S_i$ for some $i$. Among the bases $\Bas_j$, only $\Bas_i$ can contribute to the charge of $e$, giving at most $\cng(\Bas_i)$ occurrences. Moreover, $e$ can occur in $\widehat A$ only through $D_i(A)$. Since $D_i(A)=\varnothing$ when $A\cap S_i=\varnothing$, the number of these additional occurrences is at most the number of members of $\Bas$ meeting $S_i$. Each such member is counted at least once in $\sum_{s\in S_i}\ch_{\Bas}(s)$, so this number is at most $3q_i$. Therefore the charge of $e$ in $\Bas'$ is at most $\cng(\Bas_i)+3q_i$.\end{proof}

\begin{proof}[Proof of \cref{thm:main-regular-set}]
We argue by induction on $m$, with the bounded cases absorbed by increasing $C_E$. By Seymour's decomposition theorem~\cite{Seymour1980}, in the unordered form of~\cite[Lemma~5.7]{GurjarVishnoi2021}, $M$ has a decomposition tree $T$ whose nodes are graphic, cographic, or isomorphic to $R_{10}$, such that adjacent node matroids meet in a separator of size $0$, $1$, or $3$, nonadjacent node matroids are disjoint, and every subtree represents the corresponding iterated sum.

Give each node weight equal to the number of elements of $E(M)$ occurring only in that node, and choose a weighted centroid $v$ of the tree $T$. Let $B$ be the matroid at $v$, let $T_1,\ldots,T_t$ be the components of $T-v$, and let $M_i$ be the iterated sum of the matroids indexed by $V(T_i)$, with $S_i:=E(B)\cap E(M_i)$. Put $m_i:=|E(M_i)\setminus S_i|$ and omit indices with $m_i=0$, for which $S_i=\varnothing$. Reindex so that $m_1\ge\cdots\ge m_t$. Let $m_0:=|E(B)\cap E(M)|$. Every element of $M$ is counted exactly once, either by $m_0$ or by one of the numbers $m_1,\ldots,m_t$, so $m_0+\sum_{j=1}^t m_j=m$. The ordering $m_1\ge\cdots\ge m_t$ ensures that $m_j\ge m_i$ whenever $j\le i$. Hence
\begin{equation}\label{eq:ordered-sizes}
i m_i\le\sum_{j=1}^{i}m_j
\le\sum_{j=1}^{t}m_j=m-m_0\le m.
\end{equation}

Also,
\begin{equation}\label{eq:basic-size}
|E(B)|=m_0+\sum_{i=1}^t|S_i|\le m_0+3t\le m_0+3\sum_{i=1}^t m_i\le3m.
\end{equation}

Order $E(B)$ by listing first the elements of $S_1$, then those of $S_2$, and so on, followed by the elements of $E(B)\cap E(M)$. Choose a basis $\Bas$ of $\Z(B)$ as in \cref{lem:priority}, and define $q_i$ as in \cref{lem:star}. After increasing an absolute constant $C'$ if necessary, we have \begin{equation}\label{eq:priority-charges}
\begin{aligned}
q_i&\le C'(1+\log i),\\
\ch_{\Bas}(e)&\le C'\log m
\end{aligned}
\end{equation}
for every $i$ and every element $e\in E(B)\cap E(M)$. Here the elements of $S_i$ occur among the first $3i$ positions, while \eqref{eq:basic-size} controls the remaining positions.

Since $m_i$ is the total weight of $T_i$ and $v$ is a weighted centroid, we have
\begin{equation}\label{eq:centroid-size}
m_i\le m/2<m.
\end{equation}
If $m_i\ge2$, induction gives a circuit basis $\Bas_i$ of $M_i\setminus S_i$ of congestion at most $C_E\log m_i$. By \cref{lem:star} and \eqref{eq:priority-charges}, every $e\in E(M_i)\setminus S_i$ then has charge at most $C_E\log m_i+3C'(1+\log i)$. By \eqref{eq:ordered-sizes} and \eqref{eq:centroid-size}, we have $i\le m/m_i$ and $m/m_i\ge2$, so $1+\log i\le2\log(m/m_i)$. Choosing $C_E\ge6C'$ makes the preceding charge at most
\[
C_E\log m_i+C_E\log(m/m_i)=C_E\log m.
\]
If $m_i=1$, then $M_i\setminus S_i$ has a circuit basis of congestion at most one, and the same conclusion follows after one further enlargement of $C_E$. The elements of $E(B)\cap E(M)$ satisfy the desired bound by \eqref{eq:priority-charges}. Finally replace the resulting cycle basis by a circuit basis as above.
\end{proof}

Suppose that $N$ is regular and that either $N$ or $N^*$ is simple of rank $k\ge2$. Heller's extremal theorem for unimodular matrices~\cite[Theorem~4.2]{Heller1957} gives $|E(N)|\le\binom{k+1}{2}\le k^2$, since $N$ and $N^*$ have the same ground set. Applying \cref{thm:main-regular-set} to $N$, we obtain
\begin{equation}\label{eq:heller-reduction}
\bn(N)\le2C_E\log k.
\end{equation}

\begin{proof}[Proof of \cref{thm:main-regular}]
Let $d=\dim(\Z(M))\ge2$. Coloops belong to no cycle and may be deleted without changing either $d$ or the basis number. If $\{e,f\}$ is a series pair, every binary cycle meets this cocircuit evenly, so its $e$- and $f$-coordinates agree. Projection onto $E(M)\setminus\{e\}$ is therefore an isomorphism $\Z(M)\longrightarrow\Z(M/e)$ that preserves both $d$ and the optimal congestion. Repeating these operations gives a cosimple regular matroid $N$ with $\dim(\Z(N))=d$ and $\bn(N)=\bn(M)$. The dual $N^*$ is a simple regular matroid of rank $d$. Thus \eqref{eq:heller-reduction} gives $\bn(M)\le2C_E\log d$, proving the dimension bound with $C_{\rm dim}\ge2C_E$.
\end{proof}


\begin{proof}[Proof of \cref{thm:main-regular-rank}]
Let $r=r(M)\ge2$ and let $N$ be the simplification of $M$, obtained by deleting loops and retaining one element from each parallel class. Restoring a loop adds an independent one-element circuit. Suppose that $Q'$ is obtained from $Q$ by adding $f$ parallel to a nonloop $e$, and let $\Bas$ be a circuit basis of $\Z(Q)$. Put $t:=\ch_{\Bas}(e)$ and choose $\mathcal A\subseteq\{C\in\Bas:e\in C\}$ with $|\mathcal A|=\lfloor t/2\rfloor$. Put $\Bas':=(\Bas\setminus\mathcal A)\cup\{(C\setminus\{e\})\cup\{f\}:C\in\mathcal A\}$. The linear map $\pi:\Z(Q')\to\Z(Q)$ defined by $\pi(x)_e=x_e+x_f$ and $\pi(x)_g=x_g$ for $g\in E(Q)\setminus\{e\}$ maps $\Bas'$ bijectively to $\Bas$ and has kernel spanned by $\{e,f\}$. Hence $\Bas'\cup\{\{e,f\}\}$ is a circuit basis of $\Z(Q')$. The new charges at $e$ and $f$ are at most $\lceil t/2\rceil+1\le\max\{t,2\}$, and all other charges are unchanged. Consequently $\bn(M)\le\max\{\bn(N),2\}$. Since $N$ is simple of rank $r$, \eqref{eq:heller-reduction} gives $\bn(M)\le\max\{2C_E\log r,2\}=O(\log r)$.
\end{proof}

\section{Binary matroids}\label{sec:binary}

\begin{observation}
Let $M$ be a binary matroid with $m$ elements, rank $r$, and $d=\dim(\Z(M))$. We have $\bn(M)\le d$, $\bn(M)\leq m$ and $\bn(M)\le r+1$. 
\end{observation}
\begin{proof}
Since a circuit basis has $d$ members we clearly have $\bn(M)\le d$. Further, $d=\dim(\Z(M))=m-r$ and $\bn(M)\leq m$. To see $\bn(M)\le r+1$, order $E(M)$ and initialize $I=\varnothing$.
Process the elements in increasing order. If adding $e$ to $I$ would create a circuit $C$, then record $C$ and let $e_C$ be its smallest element. Otherwise, add $e$ to $I$. After each step, $I$ is independent in
$M$ and spans all elements processed so far. Hence exactly $\dim(\Z(M))=m-r$ circuits
are recorded. Each such $C$ contains its removed element $e_C$, which
occurs in no later recorded circuit, so the recorded circuits are
linearly independent in $\Z(M)$ and form a circuit basis.

Fix $e\in E(M)$, and let $I_e$ be the independent set maintained
immediately before $e$ is processed. Every recorded circuit $C$
containing $e$ is recorded at or after this step and satisfies $e_C\le e$.
If $e_C<e$, then $e_C\in I_e$. Since
distinct recorded circuits remove distinct elements, the map
$C\mapsto e_C$ injects the recorded circuits containing $e$ into
$I_e\cup\{e\}$. Thus $e$ has charge at most $|I_e|+1\le r+1$.
\end{proof}

The bound in $d$ is attained by $M_s=PG(s-1,2)^*$ for $s\ge1$, where $PG(s-1,2)$ denotes the binary projective geometry of rank $s$. The cycle space $\Z(M_s)$ has dimension $s$ and is the row space of the matrix $A$ whose columns are all nonzero vectors of $\F_2^s$. Any basis is given by the rows of $UA$ for an invertible matrix $U$. Since $UA$ still contains the all-one column, $\bn(M_s)=s$.

Linear growth in $m$ and $r$ occurs as well. For some fixed $\delta>0$ and every sufficiently large $m$, the standard random-code argument underlying the Gilbert--Varshamov bound~\cite{Varshamov1957,MacWilliamsSloane1977} gives a binary $\lceil2m/3\rceil\times m$ matrix $A$ of full row rank whose kernel has no nonzero vector of weight below $\delta m$. Its column matroid $M$ has $r=\lceil2m/3\rceil$ and $d=\lfloor m/3\rfloor$. Every circuit basis $\Bas$ satisfies $m\cng(\Bas)\ge\sum_{C\in\Bas}|C|\ge\delta md$, so $\bn(M)=\Theta(m)=\Theta(r)=\Theta(d)$. Thus the linear upper bounds are sharp in order, and none of our logarithmic bounds extends to all binary matroids.

\begin{question}\label{q:binary}
For which natural subclasses of binary matroids does a logarithmic bound in the rank or the dimension of the cycle space hold?
\end{question}

\subsection*{Statement of AI use}\label{sec:ai}

The proof was found with the help of OpenAI's GPT-5.6 Sol. The model was also used as a research aid for exploring proof strategies, locating potentially relevant literature, and assisting with the drafting and revision of the manuscript. The author independently checked the arguments and references and takes full responsibility for the mathematical correctness and final content of the paper.

\subsection*{Acknowledgements}\label{sec:ack}

We thank Stijn Cambie for pointing out the shorter proof using Freedman and Hastings from \Cref{rem:fh-shortcut} and the proof for regular matroids, both of which he found with the help of GPT-5.6. We also thank Chandra Chekuri for help with using their result~\cite{ChekuriVondrakZenklusen2010}. The author was supported through grant PID2022-137283NB-C22 funded by MICIU/AEI/10.13039/501100011033 by ERDF/EU and through the Severo Ochoa and María de Maeztu Program for Centers and Units of Excellence in R\&D (CEX2020-001084-M) and ANR project MIMETIQUE: ANR-25-CE48-4089-01.

\bibliographystyle{amsplain}
\bibliography{Logarithmic_basis_number_v20_refs}

@article{BanksSchmeichel1982,
  author  = {John Anthony Banks and Edward F. Schmeichel},
  title   = {The basis number of the {$n$}-cube},
  journal = {Journal of Combinatorial Theory, Series B},
  volume  = {33},
  year    = {1982},
  number  = {2},
  pages   = {95--100}
}

@article{BazarganiEtAl2026,
  author  = {Saman Bazargani and Therese Biedl and Prosenjit Bose and Anil Maheshwari and Babak Miraftab},
  title   = {The basis number of 1-planar graphs},
  journal = {Annals of Combinatorics},
  year    = {2026},
  note    = {DOI: 10.1007/s00026-025-00804-8; preprint arXiv:2412.18595}
}

@inproceedings{ChekuriVondrakZenklusen2010,
  author    = {Chandra Chekuri and Jan Vondr\'{a}k and Rico Zenklusen},
  title     = {Dependent randomized rounding via exchange properties of combinatorial structures},
  booktitle = {51st Annual IEEE Symposium on Foundations of Computer Science (FOCS 2010)},
  publisher = {IEEE Computer Society},
  year      = {2010},
  pages     = {575--584},
  note      = {Full version: arXiv:0909.4348, under the title \emph{Dependent Randomized Rounding for Matroid Polytopes and Applications}; theorem and section numbering cited in the text refers to the full version}
}

@article{FreedmanHastings2021,
  author  = {Michael Freedman and Matthew B. Hastings},
  title   = {Building manifolds from quantum codes},
  journal = {Geometric and Functional Analysis},
  volume  = {31},
  year    = {2021},
  number  = {4},
  pages   = {855--894},
  note    = {Preprint: arXiv:2012.02249}
}

@misc{GenietGiocanti2026,
  author = {Colin Geniet and Ugo Giocanti},
  title  = {Basis number of graphs excluding minors},
  year   = {2026},
  note   = {arXiv:2601.05195v3}
}

@article{GurjarVishnoi2021,
  author  = {Rohit Gurjar and Nisheeth K. Vishnoi},
  title   = {On the number of circuits in regular matroids (with connections to lattices and codes)},
  journal = {SIAM Journal on Discrete Mathematics},
  volume  = {35},
  year    = {2021},
  number  = {3},
  pages   = {1688--1705},
  note    = {Preprint: arXiv:1807.05164}
}

@article{Heller1957,
  author  = {I. Heller},
  title   = {On linear systems with integral valued solutions},
  journal = {Pacific Journal of Mathematics},
  volume  = {7},
  year    = {1957},
  number  = {3},
  pages   = {1351--1364}
}

@article{KavithaEtAl2009,
  author  = {Telikepalli Kavitha and Christian Liebchen and Kurt Mehlhorn and Dimitrios Michail and Romeo Rizzi and Torsten Ueckerdt and Katharina A. Zweig},
  title   = {Cycle bases in graphs: Characterization, algorithms, complexity, and applications},
  journal = {Computer Science Review},
  volume  = {3},
  year    = {2009},
  number  = {4},
  pages   = {199--243}
}

@article{MacLane1937,
  author  = {Saunders Mac Lane},
  title   = {A combinatorial condition for planar graphs},
  journal = {Fundamenta Mathematicae},
  volume  = {28},
  year    = {1937},
  pages   = {22--32}
}

@book{MacWilliamsSloane1977,
  author    = {F. J. MacWilliams and N. J. A. Sloane},
  title     = {{The Theory of Error-Correcting Codes}},
  series    = {North-Holland Mathematical Library},
  volume    = {16},
  publisher = {North-Holland Publishing Company},
  address   = {Amsterdam},
  year      = {1977}
}

@article{LehnerMiraftab2026,
  author  = {Florian Lehner and Babak Miraftab},
  title   = {Sparse cycle bases for graphs with bounded genus},
  journal = {European Journal of Combinatorics},
  volume  = {135},
  year    = {2026},
  pages   = {104371}
}

@article{Longyear1980,
  author  = {Judith Q. Longyear},
  title   = {The circuit basis in binary matroids},
  journal = {Journal of Number Theory},
  volume  = {12},
  year    = {1980},
  number  = {1},
  pages   = {71--76}
}

@article{MilgramUngar1977,
  author  = {Martin Milgram and Peter Ungar},
  title   = {Bounds for the genus of graphs with given {Betti} number},
  journal = {Journal of Combinatorial Theory, Series B},
  volume  = {23},
  year    = {1977},
  number  = {2--3},
  pages   = {227--233}
}

@misc{MiraftabMorinYuditsky2026,
  author = {Babak Miraftab and Pat Morin and Yelena Yuditsky},
  title  = {Basis number and pathwidth},
  year   = {2026},
  note   = {arXiv:2601.14095}
}

@book{MoharThomassen2001,
  author    = {Bojan Mohar and Carsten Thomassen},
  title     = {Graphs on surfaces},
  series    = {Johns Hopkins Studies in the Mathematical Sciences},
  publisher = {Johns Hopkins University Press},
  address   = {Baltimore, MD},
  year      = {2001}
}

@article{RichterShank1984,
  author  = {B. Richter and H. Shank},
  title   = {The cycle space of an embedded graph},
  journal = {Journal of Graph Theory},
  volume  = {8},
  year    = {1984},
  number  = {3},
  pages   = {365--369}
}

@article{Rizzi2009,
  author  = {Romeo Rizzi},
  title   = {Minimum weakly fundamental cycle bases are hard to find},
  journal = {Algorithmica},
  volume  = {53},
  year    = {2009},
  number  = {3},
  pages   = {402--424}
}

@article{Schmeichel1981,
  author  = {Edward F. Schmeichel},
  title   = {The basis number of a graph},
  journal = {Journal of Combinatorial Theory, Series B},
  volume  = {30},
  year    = {1981},
  number  = {2},
  pages   = {123--129}
}

@article{Seymour1980,
  author  = {P. D. Seymour},
  title   = {Decomposition of regular matroids},
  journal = {Journal of Combinatorial Theory, Series B},
  volume  = {28},
  year    = {1980},
  number  = {3},
  pages   = {305--359}
}

@article{Varshamov1957,
  author  = {R. R. Varshamov},
  title   = {Estimate of the number of signals in error correcting codes},
  journal = {Doklady Akademii Nauk SSSR},
  volume  = {117},
  year    = {1957},
  number  = {5},
  pages   = {739--741}
}

@inproceedings{WangIrani2026,
  author    = {Fan Wang and Sandy Irani},
  title     = {Cycle basis algorithms for reducing maximum edge participation},
  booktitle = {24th International Symposium on Experimental Algorithms (SEA 2026)},
  series    = {Leibniz International Proceedings in Informatics (LIPIcs)},
  volume    = {371},
  publisher = {Schloss Dagstuhl--Leibniz-Zentrum f\"ur Informatik},
  year      = {2026},
  pages     = {27:1--27:20}
}

@article{Wargo1996,
  author  = {Lawrence Wargo},
  title   = {Regular matroids with every circuit basis fundamental},
  journal = {Journal of Graph Theory},
  volume  = {21},
  year    = {1996},
  number  = {3},
  pages   = {327--334}
}

@article{Weiss1984,
  author  = {Alfred Weiss},
  title   = {Girths of bipartite sextet graphs},
  journal = {Combinatorica},
  volume  = {4},
  year    = {1984},
  number  = {2--3},
  pages   = {241--245}
}

@misc{LongPettieSaranurak2026,
  author = {Yaowei Long and Seth Pettie and Thatchaphol Saranurak},
  title  = {Deterministic Edge-Fault-Tolerant Connectivity Labeling Schemes with Nearly Optimal Label Size},
  year   = {2026},
  note   = {arXiv:2609.09031},
  doi    = {10.48550/arXiv.2609.09031},
  url    = {https://arxiv.org/abs/2609.09031}
}

@article{Welsh1969,
  author  = {D. J. A. Welsh},
  title   = {On the hyperplanes of a matroid},
  journal = {Proceedings of the Cambridge Philosophical Society},
  volume  = {65},
  number  = {1},
  pages   = {11--18},
  year    = {1969},
  doi     = {10.1017/S0305004100044017}
}

@misc{Knauer2026v1,
  author = {Kolja Knauer},
  title  = {Logarithmic basis number of graphs},
  year   = {2026},
  note   = {First arXiv version, 2 September 2026;
            \href{https://arxiv.org/abs/2609.02080v1}{arXiv:2609.02080v1}},
  url    = {https://arxiv.org/abs/2609.02080v1}
}

\end{document}